\documentclass[12pt,a4paper]{article}
\usepackage[utf8x]{inputenc}
\usepackage{ucs}
\usepackage{amsmath}
\usepackage{amsfonts}
\usepackage{amssymb}
\usepackage{amsthm}
\usepackage{mathrsfs}
\usepackage{fancyhdr}
\usepackage{graphicx}
\usepackage{xcolor}
\usepackage{arabtex}
\usepackage{framed}
\usepackage[top=2cm,bottom=2cm,margin=2cm]{geometry}
\usepackage{arabtex}
\usepackage{cancel}

\usepackage[dvipdfm,pdfstartview=FitH,colorlinks=true,linkcolor=blue,urlcolor=cyan,%
filecolor=green,citecolor=red]{hyperref}

\title{On the representation of an even perfect number as the sum of a limited number of cubes}
\author{\sc Bakir FARHI \\
Department of Mathematics \\
University of B\'ejaia \\
Algeria \\[1mm]
\href{mailto:bakir.farhi@gmail.com}{bakir.farhi@gmail.com} \\[1mm]
\url{http://www.bakir-farhi.net}
}
\date{}

\newtheorem{thm}{Theorem}

\newtheorem{conj}[thm]{Conjecture}

\newtheorem*{thmn}{Theorem}

\def\N{\mathbb{N}}

\begin{document}
\maketitle
\begin{abstract}
\noindent The aim of this note is to show that any even perfect number, other than $6$, can be written as the sum of 5 cubes of natural numbers. We also conjecture that any even perfect number, other than $6$, can be written as the sum of only 3 cubes of natural numbers.
\end{abstract}
\section{Introduction}

A perfect number is a positive integer that is equal to the sum of its proper positive divisors. The sequence of perfect numbers begins with the following numbers:
$$
6 ~;~ 28 ~;~ 496 ~;~ 8128 ~;~ 33550336 ~;~ 8589869056.
$$
There are more than twenty five centuries that mathematicians discovered the perfect numbers and began to be interested in their study. The first study of these numbers is probably due to the philosopher-mathematicians of the Pythagorean school (about 500 B.C). Eucild (about 300 B.C) showed that if $(2^n - 1)$ is prime then the number $2^{n - 1} (2^n - 1)$ is perfect. Around 100 AD, the Neopythagorean Nicomachus of Gerasa gave in his \emph{Introduction to Arithmetic} a classification of the positive integers based on the concept of perfect numbers; there are three classes: deficient numbers, perfect numbers and abundant numbers (according as the sum of the all proper positive divisors of a given positive integer $n$ is lower than, equal or greater than $n$). Nicomachus only listed the first four perfect numbers. Certainly, those are the only known perfect numbers by the Greeks.

After the Greek mathematicians, it was the Arabic mathematicians who were interested in the study of the perfect numbers and among them, we can mention Thabit Ibn Qurra, Al-Baghdadi, Ibn Al-Haytham, Al-Antaki, Ibn Al-Banna and many others. It is not possible to list the all Arabic works on this area but we just note that Ibn Al-Haytham announced and attempted to prove the converse of Euclid's theorem cited above for the even perfect numbers (see, e.g., \cite{ras}).

In the 17\textsuperscript{th} century, the study of the perfect numbers has been pursued by several mathematicians, among which Descarte, Frenicle, Fermat, Mersenne, Wolff and many others (see, e.g., \cite[Chap 1]{dic}). However, significant results was only obtained in the 18\textsuperscript{th} century by Euler who proved in 1747 the converse of Euclid's theorem for the even perfect numbers (see, e.g., \cite[Chap IV]{sie}). So, we have the following:
\begin{thmn}[Euclid-Euler]
An even positive integer $N$ is a perfect number if and only if it can be written as:
$$
N = 2^{p - 1} \left(2^p - 1\right) ,
$$
where $p$ and $(2^p - 1)$ are both primes.
\end{thmn}
Concerning the odd perfect numbers, even if they were studied by several mathematicians, we don't know up to now if there is anyone. The conjecture of odd perfect numbers states that such numbers do not exist (see, e.g., \cite{guy,san}). This conjecture is probably the oldest open problem in Number Theory and it is the target of many current researchers in this area, who follow the path lead before them by Euler, Peirce, Servais, Sylvester and many others.

Further, using the form of an even perfect number given by Euler's theorem, it is easy to show the following well-known result due to the historian of Mathematics T. L. Heath:
\begin{thmn}[T. L. Heath]
Any even perfect number, other than $6$, can be written as the sum of consecutive odd cubes of natural numbers beginning with $1$.
\end{thmn}
We just note that the proof of this theorem is based on the identity:
$$
1^3 + 3^3 + 5^3 + \dots + (2 n - 1)^3 = n^2 \left(2 n^2 - 1\right) ,
$$
which is valid for any positive integer $n$ and which we can easily verify by induction.

Looking at the above theorem, it is naturally to ask about the smallest positive integer $r$ for which any even perfect number can be written as the sum of $r$ cubes of natural numbers. In the next section, we give the result we have obtained in this direction.

\section{The result}
\begin{thm}
Any even perfect number, other than $6$, can be written as the sum of 5 cubes of natural numbers.
\end{thm}
\begin{proof}
The proof is based on the following identity:
\begin{equation}\label{eq7}
2 n^6 - 2 ~=~ \left(n^2 + n - 1\right)^3 + \left(n^2 - n - 1\right)^3 
\end{equation}
which holds for any natural number $n$. \\
Now, let $N$ be an even perfect number greater than $6$. By Euler's theorem, $N$ can be written as $N = 2^{p - 1} (2^p - 1)$, where $p$ and $(2^p - 1)$ are both prime numbers. Because $N > 6$, we have $p > 2$. For $p = 3$, we get $N = 28 = 1^3 + 3^3$, which is a sum of two cubes of natural numbers and so is also a sum of 5 cubes of natural numbers (by completing the sum by $0$'s). For the following, assume that $p > 3$. So $p$ has one of the two forms: $p = 6 k + 1$ or $p = 6 k + 5$ ($k \in \N$). \\[1mm]
\underline{1\textsuperscript{st} case:} (if $p = 6 k + 1$ for some $k \in \N$) \\[1mm]
In this case, we have $N = 2^{p - 1} (2^p - 1) = 2^{6 k} (2^{6 k + 1} - 1)$. Taking $n = 2^k$ in \eqref{eq7}, we get $2^{6 k + 1} - 2 = a^3 + b^3$, with $a = n^2 + n - 1$ and $b = n^2 - n - 1$. Hence:
$$
N ~=~ 2^{6 k} \left(2^{6 k + 1} - 1\right) ~=~ 2^{6 k} \left(a^3 + b^3 + 1\right) ~=~ \left(2^{2 k} a\right)^3 + \left(2^{2 k} b\right)^3 + \left(2^{2 k}\right)^3 ,
$$
which is a sum of 3 cubes of natural numbers and so is also a sum of 5 cubes of natural numbers (by completing the sum by $0$'s). \\[1mm]
\underline{2\textsuperscript{nd} case:} (if $p = 6 k + 5$ for some $k \in \N$) \\[1mm]
In this case, we have:
$$
N ~=~ 2^{p - 1} \left(2^p - 1\right) ~=~ 2^{6 k + 4} \left(2^{6 k + 5} - 1\right) ~=~ 2^{6 k + 3} \left(2^{6 k + 6} - 2\right) ~=~ 2^{6 k + 3} \left(64 \cdot 2^{6 k} - 2\right) .
$$
Since $64 = 3^3 + 3^3 + 2^3 + 2$, it follows that:
\begin{eqnarray}
N & = & 2^{6 k + 3} \Big((3^3 + 3^3 + 2^3 + 2) 2^{6 k} - 2\Big) \notag \\
& = & (2^{2 k + 1})^3 \Big((3 \cdot 2^{2 k})^3 + (3 \cdot 2^{2 k})^3 + (2 \cdot 2^{2 k})^3 + (2 \cdot 2^{6 k} - 2)\Big) \label{eq8}
\end{eqnarray}
Next, taking $n = 2^k$ in \eqref{eq7}, we get $2 \cdot 2^{6 k} - 2 = a^3 + b^3$ (with $a = n^2 + n - 1$ and $b = n^2 - n - 1$), which when reported in \eqref{eq8} gives:
\begin{eqnarray*}
N & = & (2^{2 k + 1})^3 \Big((3 \cdot 2^{2 k})^3 + (3 \cdot 2^{2 k})^3 + (2 \cdot 2^{2 k})^3 + a^3 + b^3\Big) \\
& = & (3 \cdot 2^{4 k + 1})^3 + (3 \cdot 2^{4 k + 1})^3 + (2 \cdot 2^{4 k + 1})^3 + (2^{2 k + 1} a)^3 + (2^{2 k + 1} b)^3 ,
\end{eqnarray*}
which is (as required) a sum of 5 cubes of natural numbers. This achieves the proof.
\end{proof}
\noindent\textbf{Remark:} It is easy to show that apart from the number $28 = 1^3 + 3^3$, none of even perfect numbers can be written as the sum of 2 cubes of natural numbers. But, up to the moment, we don't find any example of an even perfect number (other than $6$) which cannot be written as the sum of 3 cubes of natural numbers. For the first even perfect numbers, we have the following representations as a sum of 3 cubes:
$$
\begin{array}{rclcl}
28 & = & 2^2 (2^3 - 1) & = & 0^3 + 1^3 + 3^3 \\
496 & = & 2^4 (2^5 - 1) & = & 4^3 + 6^3 + 6^3 \\
8128 & = & 2^6 (2^7 - 1) & = & 4^3 + 4^3 + 20^3 \\
33550336 & = & 2^{12} (2^{13} - 1) & = & 16^3 + 176^3 + 304^3 \\
8589869056 & = & 2^{16} (2^{17} - 1) & = & 720^3 + 1336^3 + 1800^3 .
\end{array}
$$
So, we conjecture the following:
\begin{conj}
Any even perfect number, other than $6$, can be written as the sum of 3 cubes of natural numbers.
\end{conj}

\end{document}